\documentclass{amsart}

\usepackage[english]{babel}
\usepackage[applemac]{inputenc}
\usepackage[left=3cm,right=3cm,top=2.5cm,bottom=2.5cm]{geometry}
\usepackage{amsmath,amsthm, amssymb,amsfonts,nicefrac}
\usepackage[all]{xy}
\usepackage{amsmath}
\usepackage{graphicx}
\usepackage{textcomp}
\usepackage[usenames, dvipsnames]{color}
\usepackage{graphpap}
\usepackage{verbatim}

\usepackage[colorlinks,  citecolor = ForestGreen, linkcolor = NavyBlue]{hyperref}
\usepackage[T1]{fontenc} 
\usepackage{caption}

\newcommand{\Z}{\ensuremath{\mathbb{Z}}}

\newcommand{\PP}{\ensuremath{\mathbb{P}}}
\newcommand{\A}{\ensuremath{\mathbb{A}}}

\newcommand{\JJ}{\ensuremath{\mathcal{J}}}

\newcommand\kk{\mathbf{k}}

\DeclareMathOperator{\Aut}{Aut}
\DeclareMathOperator{\Bir}{Bir}
\DeclareMathOperator{\GL}{GL}

\DeclareMathOperator{\Id}{Id}
\DeclareMathOperator{\Cr}{Cr}

\DeclareMathOperator{\PGL}{PGL}
\DeclareMathOperator{\id}{id}
\DeclareMathOperator{\Sym}{Sym}
\DeclareMathOperator{\On}{O}

\newcommand\mapsfrom{\mathrel{\reflectbox{\ensuremath{\longmapsto}}}}
\def\dashmapsto{\mapstochar\dashrightarrow}

\newcommand\SBa[1][\scalebox{0.8}]{#1}

   \def\dashmapsto{\mathrel{\mapstochar\dashrightarrow}}

\newtheorem{thm}{\bf Theorem}[section]
\newtheorem*{Mthm}{\bf Main Theorem}
\newtheorem*{thm*}{\bf Theorem}

\newtheorem{lem}[thm]{\bf Lemma}
\newtheorem{prop}[thm]{\bf Proposition}

\theoremstyle{definition}
\newtheorem{defi}[thm]{\bf Definition}
\newtheorem{rmk}[thm]{\bf Remark}

\newtheorem{notation}[thm]{\bf Notation}

\author{Christian Urech and Susanna Zimmermann}
\title{A new presentation of the plane Cremona group}
\date{\today}

\address{Christian Urech\\ Departement Mathematik und Informatik\\ Spiegelgasse 1\\ 4051 Basel\\ Switzerland}
\email{christian.urech@gmail.com}
\address{Susanna Zimmermann\\ Facult\'e des sciences\\ 2 Bd de Lavoisier\\ 49045 Angers cedex 01\\ France}
\email{zimmermann@math.univ-angers.fr}

\thanks{During this work, both authors were partially supported by the Swiss National Science Foundation.}

\begin{document}
\maketitle

\begin{abstract}
We give a presentation of the plane Cremona group over an algebraically closed field with respect to the generators given by the Theorem of Noether and Castelnuovo. This presentation is particularly simple and can be used for explicit calculations. 
\end{abstract}

\section{Introduction} 
\subsection{Main result}
Let $\kk$ be an algebraically closed field. The plane Cremona group $\Bir(\PP^2)$ is the group of birational transformations of the projective plane $\PP^2=\PP^2_\kk$. It was intensely studied by classical algebraic geometers and has attracted again considerable attention in the last decades. 

If we fix homogeneous coordinates $[x:y:z]$ of $\PP^2$, every element $f\in\Bir(\PP^2)$ is given by 
\[
f\colon [x:y:z]\dashmapsto [f_0(x,y,z):f_1(x,y,z):f_2(x,y,z)],
\]
where  the $f_i$ are homogeneous polynomials of the same degree, the {\it degree of $f$}. We will write $f=[f_0(x,y,z):f_1(x,y,z):f_2(x,y,z)]$.
One of the main classical results is the Theorem of Noether and Castelnuovo  (\cite{C01}), which states that $\Bir(\PP^2)$ is generated by the linear group $\Aut(\PP^2)=\PGL_3(\kk)$ and the standard quadratic involution $\sigma:=[yz:xz:xy]$.

Throughout the last century, various presentations of the plane Cremona group have been given (see Section \ref{overview}). Each one of these displays a particularly interesting and beautiful aspect of the Cremona group. The aim of this article is to give a presentation of the plane Cremona group that uses the generators given by the Theorem of Noether and Castelnuovo. Another advantage of our presentation is that it is particularly simple and in many cases easy to use for specific calculations. We will illustrate this in Section \ref{gizhom} with an example that is due to Gizatullin.

Denote by $D_2\subset\mathrm{PGL}_3(\kk)$ the subgroup of diagonal automorphisms and by $S_3\subset\mathrm{PGL}_3(\kk)$ the symmetric group of order $6$ acting on $\PP^2$ by coordinate permutations.

\begin{Mthm}\label{thm:relations}
Let $\kk$ be an algebraically closed field. The Cremona group $\Bir(\PP^2)$ is isomorphic to
\[\Bir(\PP^2)\simeq\big\langle\sigma, \mathrm{PGL}_3(\kk)\mid (\ref{rel1})-(\ref{rel5}) \ \big\rangle\]
\begin{enumerate}
\item\label{rel1} $g_1g_2g_3^{-1}=\mathrm{id}$ for all $g_1, g_2, g_3\in\mathrm{PGL}_3(\kk)$ such that $g_1g_2=g_3$.
\item\label{rel2} $\sigma^2=\mathrm{id}$,
\item\label{rel3} $\sigma\tau(\tau\sigma)^{-1}=\mathrm{id}$ for all $\tau\in S_3$, 
\item\label{rel4} $\sigma d\sigma d=\mathrm{id}$ for all diagonal automorphisms $d\in D_2$, 
\item\label{rel5} $(\sigma h)^3=\mathrm{id}$, where $h=[z-x:z-y:z]$.
\end{enumerate}
\end{Mthm}

Observe that the relations (\ref{rel2}) to (\ref{rel4}) occur in the group $\Aut(\kk^*\times\kk^*)$, which is given by the group of monomial transformations $\GL_2(\Z)\ltimes D_2$. 
Relation (\ref{rel5}) is a relation from the group 
\[\Aut(\PP^1\times\PP^1)^0\simeq\mathrm{PGL}_2(\kk)\times\mathrm{PGL}_2(\kk)\]
which we consider as a subgroup of $\Bir(\PP^2)$ by conjugation with the birational equivalence  $\varpi\colon\PP^1\times\PP^1\dashrightarrow\PP^2$, given by $([u_0:u_1], [v_0:v_1])\dashmapsto [u_1v_0:u_0v_1:u_1v_1]$.\par
The idea of the proof of the Main Theorem is the same as in \cite{I84,B12,Z16}. We study linear systems of compositions of birational transformations and use the presentation of the Cremona group given by Blanc in \cite{B12}.


\subsection{Previous presentations}\label{overview}
The aim of this section is to give an overview over previous presentations of the plane Cremona group. The first presentation was given by Gizatullin in 1982 (\cite{G82}, see also \cite{G90}):

\begin{thm}[{\cite[Theorem 10.7]{G82}}]
The Cremona group $\Bir(\PP^2)$ is generated by the set $Q$ of all quadratic transformations and the relations in $\Bir(\PP^2)$ are consequences of relations of the form $q_1q_2q_3=\id$, where $q_1,q_2, q_3$ are quadratic transformations, i.e. we have the presentation:
\[
\Bir(\PP^2)=\langle Q\mid q_1q_2q_3=\id \text{ for all } q_1, q_2, q_3\in Q \text{ such that } q_1q_2q_3=\id \text{ in }\Bir(\PP^2)\rangle.
\]
\end{thm}

\begin{defi}
We denote by $\JJ\subset\Bir(\PP^2)$ the group of transformations preserving the pencil of lines through $[1:0:0]$ and call it the {\it de Jonqui\`eres group}.
\end{defi}
With respect to affine coordinates $(x,y)$, the de Jonqui\`eres can be described by: 
\begin{align*}\JJ=&\left\{(x,y)\dashmapsto\left(\frac{ax+b}{cx+d},\frac{\alpha(x)y+\beta(x)}{\gamma(x)y+\delta(x)}\right)\middle\vert \  \left(\begin{matrix}a&b\\ c&d\end{matrix}\right)\in\mathrm{PGL}_2(\kk),\  \left(\begin{matrix}\alpha(x)&\beta(x)\\ \gamma(x)&\delta(x)\end{matrix}\right)\in\mathrm{PGL}_2(\kk(x))\right\}\\
\simeq&\ \mathrm{PGL}_2(\kk)\ltimes\mathrm{PGL}_2(\kk(x))
\end{align*}

By the Theorem of Noether-Castelnuovo, $\Bir(\PP^2)$ is generated by $\JJ$ and $\PGL_3(\kk)$, since $\sigma$ is an element of $\JJ$. The following theorem shows that, in a certain way, most of the relations in $\Bir(\PP^2)$ are consequences of relations within these groups:

\begin{thm}[\cite{B12}]\label{blancisk}
The Cremona group $\Bir(\PP^2)$ over an algebraically closed field is the amalgamated product of $\PGL_3(\kk)$ and $\JJ$ along their intersection, divided by the relation
\[
\sigma\tau=\tau\sigma,
\]
where $\sigma$ is the standard involution and $\tau\in\PGL_3(\kk)$ the transposition $[y:x:z]$.
\end{thm}

Note that the Cremona group does not have the structure of an amalgamated product \cite{C13}.
Theorem~\ref{blancisk} was proceeded by the following statement shown by Iskovskikh in \cite{I84}:

\begin{thm}[\cite{I84}]\label{isk}
The Cremona group $\Bir(\PP^1\times\PP^1)$ is generated by $\tau\colon (x,y)\mapsto(y,x)$ and the group $B$ of birational transformations preserving the fibration given by the first projection and the following relations form a complete system of relations:
\begin{itemize}
\item relations inside the groups $\Aut(\PP^1\times\PP^1 )$ and $B$, 
\item $(\tau \circ (\frac{1}{y},\frac{y}{x}))^3=\id$,
\item $(\tau\circ(-x,y-x))^3=\id$.
\end{itemize}
\end{thm}

A gap in the original proof of Theorem \ref{isk} had been detected and closed by Lamy (\cite{L10}).
 
A presentation of $\Bir(\PP^2)$ in the form of a generalised amalgam was given in the following statement:

\begin{thm}[\cite{W92}]
The group $\Bir(\PP^2)$ is the free product of $\PGL_3(\kk)$, $\Aut(\PP^1\times\PP^1)$ and $ \JJ$ 
amalgamated along their pairwise intersections in $\Bir(\PP^2)$.
\end{thm}

In \cite{BF13} the authors introduced the Euclidean topology on the Cremona group over a locally compact local field. With respect to this topology, $\Bir(\PP^2)$ is a Hausdorff topological group and the restriction of the Euclidean topology to any algebraic subgroup is the classical Euclidean topology. In order to show that $\Bir(\PP^2)$ is compactly presentable with respect to the Euclidean topology, Zimmermann proved the following:

\begin{thm}[\cite{Z16}]
The Cremona group $\Bir(\PP^2)$ is isomorphic to the amalgamated product of $\Aut(\PP^2)$, $\Aut(\mathbb{F}_2)$, $\Aut(\PP^1\times\PP^1)$ along their pairwise intersection in $\Bir(\PP^2)$ modulo the relation $\tau\sigma\tau\sigma$, where $\tau\in\Aut(\PP^2)$ is the coordinate permutation $\tau\colon [x:y:z]\mapsto [z:y:x]$.
\end{thm}

We would also like to mention the paper \cite{IKT94} by Iskovskikh, Kabdykairov and Tregub, in which the authors present a list of generators and relations of $\Bir(\PP^2)$ over arbitrary perfect fields.

\subsection{Gizatullin-homomorphisms between Cremona groups}\label{gizhom}
In this Section we recall a result of Gizatullin in order to illustrate how our presentation can be used for explicit calculations. Throughout Section \ref{gizhom} we assume $\kk$ to be algebraically closed and of characteristic $\neq 2$. In \cite{G90}, Gizatullin considers the following question: Can a given group-homomorphism $\varphi\colon\PGL_3(\kk)\to\PGL_{n+1}(\kk)$ be extended to a group-homomorphism $\Phi\colon\Cr_2(\kk)\to\Cr_n(\kk)$? He answers this question positively if $\varphi$ is the projective representation induced by the regular action of $\PGL_3(\kk)$ on the space of plane conics, plane cubics or plane quartics. In order to construct these homomorphisms he uses the following construction. Let $\Sym_n$ be the $\kk$-algebra of symmetric $n\times n$-matrices and define the variety $\mathbb{S}_2(n)$ to be the quotient $(\Sym_n)^3//\GL_n(\kk)$ where the regular action of $\GL_n(\kk)$ is given by $C\cdot(A_1, A_2, A_3)=(CA_1C^T, CA_2C^T, CA_3C^T)$. 

\begin{lem}\label{rational}
	The variety $\mathbb{S}_2(n)$ is a rational variety of dimension $(n+1)(n+2)/2-1$.
\end{lem}

\begin{proof}
	The dimension of $(\Sym_n)^3$ is $3n(n+1)/2$. Since the dimension of $\GL_n(\kk)$ is $n^2$ and the action of $\GL_n(\kk)$ on $(\Sym_n)^3$ has finite kernel, the dimension of $\mathbb{S}_2(n)$ is $(n+1)(n+2)/2-1$.
	
	Define the subvariety $X\subset (\Sym_n)^3$ given by triplets of the form $(\id, d, A)$, where $d\in D_2$ and consider the following maps:
	\[
	X\hookrightarrow (\Sym_n)^3\xrightarrow{\pi} \mathbb{S}_2(n).
	\]
	Let $U\subset(\Sym_n)^3$ be the open dense subset of triplets of the form $(A_1, A_2, A_3)$ such that $A_1$ and $A_2$ are invertible. We note that $\pi(X)=\pi(U)$. Indeed, every symmetric element $A_1\in\GL_n(\kk)$ is of the form $CC^T$ for some $C\in\GL_n(\kk)$ and for every symmetric matrix $A_2\in \GL_n(\kk)$ there exists an orthogonal matrix $S\in\On_n(\kk)$ such that $SA_2S^T$ is diagonal. Consider the open dense subset $V\subset X$ of elements of the form $(\id, d, A)$, where $A\in\Sym_n$ and $d=(d_{ij})$ an element in $D_2$ such that $d_{ii}\neq d_{jj}$ for all $i\neq j$. This condition ensures that the centralizer $G$ of $d$ in the orthogonal group $\On_n(\kk)$ consists only of diagonal elements. The group $G$ is therefore an abelian group of exponent $2$. A Theorem of Fischer (\cite{F15}) states that the quotient of $\A^N$ by $G$ is a rational variety. The image $\pi(V)$ is still dense in $\mathbb{S}_2(n)$ and the restriction of $\pi$ to $V$ is the quotient map of the action of $G$ on $V$ given by conjugation. The theorem of Fischer therefore implies that $ \mathbb{S}_2(n)$ is rational.
\end{proof}
Note that Lemma \ref{rational} implies that $\mathbb{S}_2(n)$ has the same dimension as the space of all plane curves of degree $n$.

A linear transformation $g=[g_1:g_2:g_3]\in\PGL_3(\kk)$ induces an automorphism on $(\Sym_n)^3$ by  $g(A_1, A_2, A_3):=(g_1(A_1, A_2, A_3), g_2(A_1, A_2, A_3), g_3(A_1, A_2, A_3))$. This automorphism commutes with the action of $\GL_n(\kk)$, so we obtain a regular action of $\PGL_3(\kk)$ on $\mathbb{S}_2(n)$. Our main theorem allows now to give a short proof of the following result of Gizatullin:

\begin{prop}[\cite{G90}]\label{giz}
	This regular action of $\PGL_3(\kk)$ extends to a rational action of $\Bir(\PP^2)$ on $\mathbb{S}_2(n)$. 
\end{prop}

\begin{proof}
	We define the birational action of the element $\sigma$ on $\mathbb{S}_2(n)$ by
	\[
	 (A_1, A_2, A_3)\dashmapsto (A_1^{-1}, A_2^{-1}, A_3^{-1}).
	\] 
	 In order to see that this indeed defines a rational action of $\Cr_2(\kk)$ on $\mathbb{S}_2(n)$ it is enough, by our main theorem, to check that the relations (\ref{rel1}) to (\ref{rel5}) are satisfied. The relations (\ref{rel1}) to (\ref{rel4}) are straightforward to check. For relation (\ref{rel5}) we calculate for general $A_1, A_2, A_3\in \Sym_n$:
	\[
	\sigma h \sigma (A_1, A_2, A_3)=((A_3^{-1}-A_1^{-1})^{-1},(A_3^{-1}-A_2^{-1})^{-1}, A_3 )
	\]
	and 
	\[
	h\sigma h (A_1, A_2, A_3)=(A_3^{-1}-(A_3-A_1)^{-1}, A_3^{-1}-(A_3-A_2)^{-1}, A_3^{-1})
	\]
	\[
	=(A_3-A_3(A_3-A_1)^{-1}A_3, A_3-A_3(A_3-A_2)^{-1}A_3, A_3)
	\]
One calculates that the two expressions are the same. Hence relation (\ref{rel5}) is satisfied.  
\end{proof}

Lemma \ref{rational} and Proposition \ref{giz} imply that there is a rational $\Cr_2(\kk)$-action on $\PP^N$ for all $N=(n+1)(n+2)/2-1$. In \cite{G90} Gizatullin shows using classical geometry that, for $n=2,3$ and $4$, the $\PGL_3(\kk)$-action on $\mathbb{S}_2(n)$ is conjugate to the $\PGL_3(\kk)$-action on the space of plane conics, plane cubics and plane quartics. It is an interesting question whether the $\PGL_3(\kk)$-actions on $\mathbb{S}_2(n)$ are conjugate to the $\PGL_3(\kk)$-actions on the space of plane curves of degree $n$. A positive answer would give  a rational action of $\Cr_2(\kk)$ on the space of all plane curves preserving degrees. It would also be interesting to look at the geometrical properties of these rational actions. The case of the rational action of $\Cr_2(\kk)$ on the space of plane conics has been studied in \cite{U16}.

\vskip\baselineskip
\noindent{\bf Acknowledgements:} The authors would like to express their warmest thanks to Serge Cantat and J\'er\'emy Blanc for many interesting discussions and for pointing out a mistake in the first version of the proof. They are also grateful to Marat Gizatullin,  Mattias Hemmig and Anne Lonjou for many discussions.

\section{Useful relations}
Throughout all the sections we work over an algebraically closed field $\kk$.
\subsection{Preliminaries}
There is a general formula for the degree of a composition of two Cremona transformations \cite[Corollary 4.2.12]{A02}, but the multiplicities of the base-points of the composition are hard to compute in general. However, if we compose an arbitrary birational map with a map of degree $2$ it is a rather straight forward calculation \cite[Proposition 4.2.5]{A02}. We will only use the formula for de Jonqui\`eres maps. For a given transformation $f\in\Bir(\PP^2)$, we denote by $m_p(f)$ the multiplicity of $f$ in the point $p$.

\begin{lem}\label{lem:composition}
Let $\tau,f\in\JJ$ be transformations of degree $2$ and $d$ respectively. Let $p_1,p_2$ be the base-points of $\tau$ different from $[1:0:0]$ and $q_1,q_2$ the base-points of $\tau^{-1}$ different from $[1:0:0]$ such that the pencil of lines through $p_i$ is sent by $\tau$ onto the pencil of lines through $q_i$. Then
\begin{align*}\deg(f\tau)&=d+1-m_{q_1}(f)-m_{q_2}(f)\\
m_{[1:0:0]}(f\tau)&=d-m_{q_1}(f)-m_{q_2}(f)=\deg(f\tau)-1\\
m_{p_i}(f\tau)&=1-m_{q_j}(f),\quad i\neq j.
\end{align*}
\end{lem}
\begin{proof}
A de Jonqui\`eres map of degree $d$ has multiplicity $d-1$ in $[1:0:0]$ and multiplicity $1$ in every other of its base-points. The claim now follows from \cite[Proposition 4.2.5]{A02}.
\end{proof}

\begin{rmk}\label{rmk:composition}
Keep the notation of Lemma~\ref{lem:composition}. Let $\Lambda$ be the linear system of $f$. Then Lemma~\ref{lem:composition} translates to:
\begin{align*}\deg(f\tau)=\deg((\tau^{-1})(\Lambda))&=d+1-m_{q_1}(\Lambda)-m_{q_2}(\Lambda)\\
m_{[1:0:0]}((\tau^{-1})(\Lambda))&=d-m_{q_1}(\Lambda)-m_{q_2}(\Lambda)=\deg((\tau^{-1})(\Lambda))-1\\
m_{p_i}((\tau^{-1})(\Lambda))&=1-m_{q_j}(\Lambda),\quad i\neq j.
\end{align*}
In particular, as the multiplicity of $\Lambda$ in a point different from $[1:0:0]$ is zero or one, we obtain
\[\deg((\tau^{-1})(\Lambda))=\begin{cases}\deg(\Lambda)+1,&m_{q_1}(\Lambda)=m_{q_2}(\Lambda)=0\\ \deg(\Lambda),&m_{q_1}(\Lambda)+m_{q_2}(\Lambda)=1\\ \deg(\Lambda)-1,&m_{q_1}(\Lambda)=m_{q_2}(\Lambda)=1.\end{cases}\]
Note as well that B\'ezout theorem implies that $[1:0:0]$ and any two other base-points of $f$ are not collinear since $[1:0:0]$ is a base-point of multiplicity $d-1$ all other base-points are of multiplicity 1 and a general member of $\Lambda$ intersects a line in $d$ points counted with multiplicity.
\end{rmk}

\begin{notation}
We use the following picture to work with relations. If $f_1,\dots,f_n\in\{\sigma\}\cup\PGL_3(\kk)$ and $f_1\cdots f_n=1$ in the group $\langle\sigma,\PGL_3(\kk)\mid(\ref{rel1})-(\ref{rel5})\rangle$, we say that the commutative diagram
\[\xymatrix{\ar@{-->}[r]^{f_n}\ar@/_1pc/[rrrr]_{\text{Id}}&\ar@{-->}[r]^{f_{n-1}}&\ar@{..>}[r]&\ar@{-->}[r]^{f_1}&}\quad\text{or}\quad\xymatrix{\ar@{-->}[r]^{f_n}\ar@/_1pc/@{-->}[rrrr]_{f_1^{-1}}&\ar@{-->}[r]^{f_{n-1}}&\ar@{..>}[r]&\ar@{-->}[r]^{f_2}&}\]
is generated by relations (\ref{rel1})--(\ref{rel5}) and that the expression $f_1\cdots f_n=1$ is generated by relations (\ref{rel1})--(\ref{rel5}).
\end{notation}

\subsection{Relations}
When dealing with quadratic maps, some relations among them appear quite naturally. We now take a closer look at a few of them and show that they are in fact generated by relations (\ref{rel1})--(\ref{rel5}). 

\begin{lem}\label{lem:deg1}
Let $g\in\PGL_3(\kk)\cap\JJ$ and suppose that the map $g':=\sigma g\sigma$ is linear. Then $g'=\sigma g\sigma$ is generated by relations $(\ref{rel1})$--$(\ref{rel4})$. In other words, the commutative diagram
\[\xymatrix{ \ar@{-->}[r]^{\sigma}\ar@/_1pc/[rrr]_{\alpha'} &\ar[r]^{g}&\ar@{-->}[r]^{\sigma}&}\]
is generated by relations $(\ref{rel1})$--$(\ref{rel4})$.
\end{lem}
\begin{proof}
The map $\sigma g\sigma$ being linear means by Lemma~\ref{lem:composition} that the base points of $\sigma g$ are the same as the base points of $\sigma$, which are $\{[1:0:0],[0:1:0],[0:0:1]\}$. As $g\in\JJ$, it fixes $[1:0:0]$ and thus permutes the points $[0:1:0],[0:0:1]$. Hence $g=d\tau$ for some $\tau\in S_3\cap\JJ$ and $d\in D_2$. Using our relations we obtain
\[\sigma g\sigma\stackrel{(\ref{rel1})}=\sigma d \tau\sigma\stackrel{(\ref{rel3}),(\ref{rel4})}=d^{-1}\tau=g'.\]
\end{proof}

\begin{lem}\label{lem:deg2}
Let $g\in\PGL_3(\kk)\cap\JJ$ such that $\deg(\sigma g\sigma)=2$. Suppose that no three of the base-points of $\sigma$ and $\sigma g$ are collinear. Then there exist $g',g''\in\PGL_3(\kk)\cap\JJ$ such that $\sigma g\sigma=g'\sigma g''$ and this expression is generated by relations $(\ref{rel1})$--$(\ref{rel5})$. In other words, the commutative diagram
\[\xymatrix{ &\ar[r]^{g}&\ar@{-->}[rd]^{\sigma}&\\
\ar@{-->}[ru]^{\sigma} \ar[r]^(.75){g''}&\ar@{-->}[r]^{\sigma}&\ar[r]^(.3){g'}& }\]
is generated by relations $(\ref{rel1})$--$(\ref{rel5})$.
\end{lem}
\begin{proof}
The assumption $\deg(\sigma g\sigma)=2$ implies by Lemma~\ref{lem:composition} that $\sigma$ and $\sigma g$ have exactly two common base-points, among them $[1:0:0]$ because $\sigma g$ and $\sigma$ are de Jonqui\`eres maps. Up to coordinate permutations, the second point is $[0:1:0]$. More precisely, there exist two coordinate permutations $\tau_1, \tau_2\in S_3\cap\JJ$ such that $\tau_1g\tau_2$ fixes the points $[1:0:0]$ and $[0:1:0]$. Hence there are $a_1, a_2, b_1, b_2, c\in\kk$ such that
\[\tau_1g\tau_2=[a_1x+a_2z:b_1y+b_2z:cz].\]
The assumption that no three of the base-points of $\sigma$ and $\sigma g$ are collinear implies that $a_2b_2\neq0$. So there exist 
two elements $d_1, d_2\in D_2$ such that $\tau_1g\tau_2=d_1hd_2$. Using relations (\ref{rel1})--(\ref{rel5}) we obtain
\begin{align*}
\sigma g\sigma=\sigma \tau_1^{-1}\tau_1g\tau_2\tau_2^{-1}\sigma\stackrel{(\ref{rel1})}= \sigma \tau_1^{-1}d_1hd_2\tau_2^{-1}\sigma&\stackrel{(\ref{rel3}),(\ref{rel4})}= \tau_1^{-1}d_1^{-1}\sigma h\sigma d_2^{-1}\tau_2^{-1}\stackrel{(\ref{rel5})}=\tau_1^{-1}d_1^{-1}h\sigma h d_2^{-1}\tau_2^{-1}.
\end{align*}
The claim follows with $g':=\tau_1^{-1}d_1^{-1}h$ and $g'':=h d_2^{-1}\tau_2^{-1}$. 
\end{proof}

\begin{lem}\label{lem:square}
Let $g_1,\dots,g_4\in\mathrm{PGL}_3(\kk)\cap\JJ$ such that $\sigma g_2\sigma g_1=g_4\sigma g_3\sigma$ is of degree $3$ and 
\begin{itemize}
\item no three of the base-points of $\sigma g_2$ and $\sigma$ are collinear
\item no three of the base-points of $\sigma g_1$ and $\sigma $ are collinear.
\end{itemize}
Then the relation $\sigma g_2\sigma g_1=g_4\sigma g_3\sigma$ is generated by relations $(\ref{rel1})$--$(\ref{rel5})$. 
\end{lem}
\begin{proof} Let $p_0:=[0:0:1]$. The maps $g_1,\dots,g_4\in\PGL_3(\kk)\cap\JJ$ fix it. By Lemma~\ref{lem:composition}, the assumption $\deg(\sigma g_2\sigma g_1)=3$ is equivalent to the maps $\sigma g_2$ and $g_1^{-1}\sigma$ having exactly one common base-point, namely $p_0$. For the same reason, the maps $g_4\sigma$ and $\sigma g_3^{-1}$ have only $p_0$ as a common base-point. The assumptions on the base-points imply that the maps $\sigma g_2\sigma g_1$ and $g_3\sigma g_1^{-1}\sigma$ have only proper base-points. Therefore, the maps $g_4\sigma g_3\sigma$ and $\sigma g_4^{-1}\sigma g_2$ have only proper base-points and we obtain that for 
\[(f,g)\in\{(g_1^{-1}\sigma,\sigma g_2), (\sigma g_1,g_3\sigma),(\sigma g_3^{-1},g_4\sigma),(g_2^{-1}\sigma,\sigma g_4^{-1})\}\]
no three base-points of the maps $f$ and $g$ are collinear. {Hence there exists a quadratic map} $\tau_1$ that has exactly two base-points in common with $\sigma g_1$ and with $g_3\sigma$. We write $\tau_1=h_1\sigma h_2$ for some $h_1, h_2\in\PGL_3(\kk)\cap\JJ$. Then $\sigma g_1 \tau_1^{-1}$ and $g_3\sigma\tau_1^{-1}$ are quadratic again. Let $h_3,\dots,h_8\in\PGL_3(\kk)\cap\JJ$ and define the quadratic maps $\tau_i:=h_{2i}\sigma h_{2i-1}\in\JJ$, $i=2,\dots,4$. We can chose  $h_3, h_4$ such that $\tau_2^{-1}\tau_1=h_3^{-1}\sigma h_4^{-1}h_2\sigma h_3=\sigma g_1$. Analogously we chose $h_4,\dots, h_8$ such that:

$\bullet$ the diagram below commutes

$\bullet$ each of the $\tau_i$ has exactly two common base-points (one of them $p_0$) with each of the two other maps departing from the same corner.
\[\xymatrix{\ar@{-->}[rr]^{\sigma g_1} \ar@{-->}[dd]_{g_3\sigma} \ar@{-->}[rd]^{\tau_1}&&\ar@{-->}[dd]^{\sigma g_2}\ar@{-->}[ld]^{\tau_2}\\
&&\\
\ar@{-->}[rr]^{g_4\sigma}\ar@{-->}[ur]^{\tau_3}&&\ar@{-->}[ul]_{\tau_4}}\]
Lemma~\ref{lem:deg2} applied to the triangles yields that each of them is generated by relations (\ref{rel1})--(\ref{rel5}). In particular, the above commutative diagram is generated by relations (\ref{rel1})--(\ref{rel5}). 
\end{proof}

\begin{rmk}\label{rmk:the linear map}
Let $a_1,a_2,b_1,b_2,c\in\kk$. The map
\[g=[a_1x+a_2z:b_1y+b_2z:cz]\in\mathrm{PGL}_3(\kk)\cap\JJ\] is a de Jonqui\'eres map. Furthermore, observe that if $a_2=b_2=0$, then $\sigma g\sigma$ is linear. Otherwise, $\sigma g\sigma$ is a quadratic map. In fact, if $a_2b_2\neq0$, all its base-points are in $\PP^2$; this corresponds to the case that is treated in Lemma \ref{lem:deg2}. If $a_2=0$, the map $\sigma g\sigma$ has a base-point infinitely near to $[1:0:0]$, and if $b_2=0$ it has a base-point infinitely near to $[0:1:0]$. Figure~\ref{fig:bp} displays the constellation of the base-points of $\sigma g$ and $\sigma$. \par
\begin{minipage}[t]{\textwidth}
\center
\def\svgwidth{.5\textwidth}
\begingroup%
  \makeatletter%
  \providecommand\color[2][]{%
    \errmessage{(Inkscape) Color is used for the text in Inkscape, but the package 'color.sty' is not loaded}%
    \renewcommand\color[2][]{}%
  }%
  \providecommand\transparent[1]{%
    \errmessage{(Inkscape) Transparency is used (non-zero) for the text in Inkscape, but the package 'transparent.sty' is not loaded}%
    \renewcommand\transparent[1]{}%
  }%
  \providecommand\rotatebox[2]{#2}%
  \ifx\svgwidth\undefined%
    \setlength{\unitlength}{1078.09529943bp}%
    \ifx\svgscale\undefined%
      \relax%
    \else%
      \setlength{\unitlength}{\unitlength * \real{\svgscale}}%
    \fi%
  \else%
    \setlength{\unitlength}{\svgwidth}%
  \fi%
  \global\let\svgwidth\undefined%
  \global\let\svgscale\undefined%
  \makeatother%
  \begin{picture}(1,0.40583394)%
    \put(-0.15064638,0.19547697){\color[rgb]{0,0,0}\makebox(0,0)[lb]{\smash{\SBa{$g^{-1}([0:0:1])$}}}}%
    \put(0,0){\includegraphics[width=\unitlength,page=1]{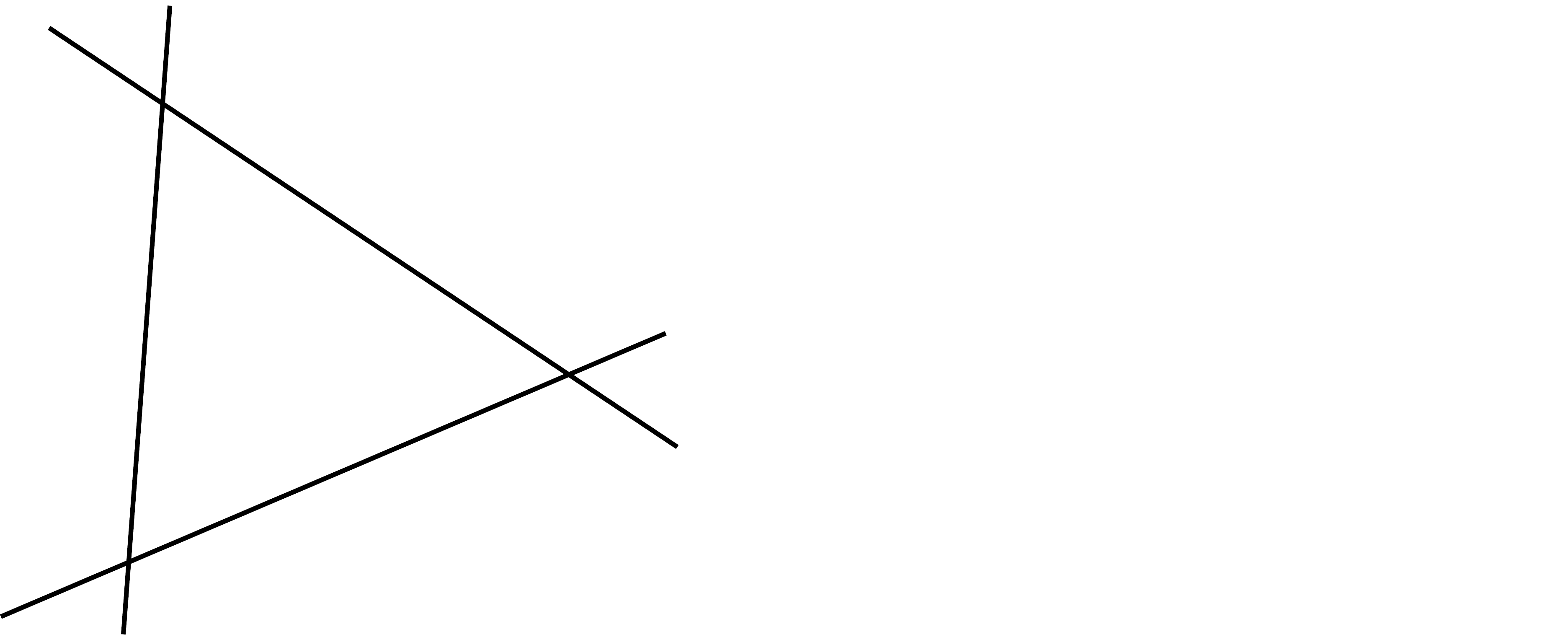}}%
    \put(0.11048432,0.01349818){\color[rgb]{0,0,0}\makebox(0,0)[lb]{\smash{\SBa{$[1:0:0]$}}}}%
    \put(0.34900019,0.20961123){\color[rgb]{0,0,0}\makebox(0,0)[lb]{\smash{\SBa{$[0:1:0]$}}}}%
    \put(0.13309916,0.34883386){\color[rgb]{0,0,0}\makebox(0,0)[lb]{\smash{\SBa{$[0:0:1]$}}}}%
    \put(0,0){\includegraphics[width=\unitlength,page=2]{fig07.pdf}}%
    \put(0.67762209,0.01703174){\color[rgb]{0,0,0}\makebox(0,0)[lb]{\smash{\SBa{$[1:0:0]$}}}}%
    \put(0.91613796,0.2131448){\color[rgb]{0,0,0}\makebox(0,0)[lb]{\smash{\SBa{$[0:1:0]$}}}}%
    \put(0.70023693,0.35236742){\color[rgb]{0,0,0}\makebox(0,0)[lb]{\smash{\SBa{$[0:0:1]$}}}}%
    \put(0,0){\includegraphics[width=\unitlength,page=3]{fig07.pdf}}%
    \put(0.83168577,0.28028262){\color[rgb]{0,0,0}\makebox(0,0)[lb]{\smash{\SBa{$g^{-1}([0:0:1])=[0:-\frac{b_2}{b_1c}:\frac{1}{c}]$}}}}%
    \put(0,0){\includegraphics[width=\unitlength,page=4]{fig07.pdf}}%
    \put(-0.17290787,0.13505299){\color[rgb]{0,0,0}\makebox(0,0)[lb]{\smash{\SBa{$=[-\frac{a_2}{a_1c}:0:\frac{1}{c}]$}}}}%
  \end{picture}%
\endgroup%

\captionof{figure}{The base-points of $\sigma g$ if $a_2\neq0$, $b_2=0$ (left) and if $a_2=0$, $b_2\neq0$ (right). } \label{fig:bp}
\end{minipage}
\end{rmk}

The following two lemmas provide a solution of how to treat these two cases. 

\begin{defi}
For a de Jonqui\`eres transformation $f\in\JJ$ and a line $l\subset\PP^2$ we define the {\it discrepancy} $N(f,l)$ as follows:
If $f(l)$ is a line, we set $N(f,l):=0$. If $f(l)=\{p\}$ is a point, then there exists a sequence of $K$ blow-ups $\pi\colon S\rightarrow \PP^2$ and an induced birational map $\hat{f}\colon \PP^2\dashrightarrow S$, such that the following diagram commutes:
\[\xymatrix{&S\ar[d]^{\pi}\\ \PP^2\ar@{-->}[r]^f \ar@{-->}[ru]^{\hat{f}}&\PP^2}\]
and such that the strict transform $\hat{f}(l)\subset S$ is a curve. We define $N(f,l)$ to be the least number $K$ of blow-ups necessary for this construction.
\end{defi}

\begin{rmk}
Let $f\in\JJ$ and let $l\subset\PP^2$ be a line. If $N(f,l)\geq1$, then the point $P:=f(l)$ is a base-point of $f^{-1}$. More precisely, if $\pi$ and $\hat{f}$ are as above, then $f^{-1}$ has a base-point in the $(N(f,l)-1)$-th neighbourhood of $f(l)$ and the exceptional divisor of this base-point is $\hat{f}(l)$. 
\end{rmk}

\begin{lem}\label{lem:discrepancy}
Let $f\in\JJ$ be a de Jonqui\`eres transformation and $l\subset\PP^2$  a line such that $N(f,l)\geq 1$, and let $g\in\PGL_3(\kk)\cap \JJ$. Denote by $\eta\colon S\rightarrow\PP^2$ the blow-up of the base-points of $\sigma g$. Then 
\begin{enumerate}
\item $N(\sigma gf,l)=N(f,l)+1$ if and only if the point $f(l)\in\PP^2$ is on a line contracted by $\sigma g$ but is not a base-point of $\sigma g$.
\item $N(\sigma gf,l)=N(f,l)-1$ if and only if the point $f(l)\in\PP^2$ is a base-point of $\sigma g$ and $(\eta^{-1}f)(l)$ is a line or a point not on the strict transform of a line contracted by $\sigma g$.
\item $N(\sigma gf,l)=N$ otherwise.
\end{enumerate}
\end{lem}
\begin{proof}
If $p:=f(l)$ is not a base-point of $\sigma g$, then its image by $\sigma g$ is a base-point of $(\sigma gf)^{-1}$ with $N(\sigma gf,l)\geq N(f,l)$. The inequality is strict if and only if $p$ is on a line contracted by $\sigma g$. 

If $p$ is a base-point of $\sigma g$ the discrepancy decreases strictly unless $\eta^{-1}f$ contracts $l$ onto an intersection point of the exceptional divisor of $p$ and the strict transform of a line contracted by $\sigma g$.
\end{proof}

\begin{lem}\label{lem:no bp 1} 
Let $g_1,\dots,g_m\in\mathrm{PGL}_3(\kk)\cap\JJ$ such that $g_m\sigma g_{m-1}\sigma\cdots \sigma g_1=\mathrm{id}$. Then there exist linear maps $h_1,\dots,h_m\in\PGL_3(\kk)\cap\JJ$ and a de Jonqui\`eres transformation $\varphi\in\JJ$ such that 
\[\varphi g_m\sigma g_{m-1}\sigma\cdots \sigma g_1\varphi^{-1}=h_m\sigma h_{m-1}\sigma\cdots h_1\sigma, \]
and such that the following properties are satisfied:
\begin{enumerate}
\item the above relation is generated by relations $(\ref{rel1})$--$(\ref{rel5})$,
\item $\deg(\sigma h_i\sigma h_{i-1}\cdots \sigma h_1)=\deg(\sigma g_i\sigma g_{i-1}\cdots \sigma g_1)$ for all $i=1,\dots,m$,
\item $(\sigma h_i\sigma h_{i-1}\cdots \sigma h_1)^{-1}$ does not have any infinitely near base-points above $[1:0:0]$ for all $i=1,\dots,m$.
\end{enumerate}
\end{lem}
\begin{proof}
Define $f_i:=\sigma g_i\cdots \sigma g_1$ for $i=1,\dots,m-1$. We will construct the $h_i$ such that  (3) is satisfied and then show that the other properties hold as well. If none of the $f_i$ has a base-point that is infinitely near to $[1:0:0]$, we set $h_i:=g_i$ for all $i$ and all the claims are trivially satisfied. Otherwise, there exist $k\geq 1$ such that $f_i$ has a base-point in the $k$-th neigbourhood of $[1:0:0]$ for some $i$. Assume $k$ to be maximal with this property. We will lower  $k$ and then proceed by induction. 
For this we pick two general points $s_0,t_0\in\PP^2$ and define
\begin{align*}&s_i:=f_i(s_0)=\sigma g_i\sigma\cdots\sigma g_1(s_0),\quad t_i:=f_i(t_0)=\sigma g_i\sigma\cdots\sigma g_1(t_0),\qquad \text{ for }i=1,\dots,m-2,\\
& s_{m-1}:=g_m\sigma g_{m-1}(s_{m-2}),\quad t_{m-1}:=g_m\sigma g_{m-1}(t_{m-2}).
\end{align*}
Since $s_0$ and $t_0$ are general points of $\PP^2$, the $s_i$ and $t_i$ are general as  well. For all $i=1,\dots,m-1$, there exist $\alpha_i\in\mathrm{PGL}_3(\kk)\cap\JJ$ that send $p_0,p_1,p_2$ onto $p_0,s_i,t_i$ respectively. The map 
\[\tau_i:=\alpha_i\sigma\alpha_i^{-1}\in\JJ\] 
is a quadratic involution with base-points $p_0,s_i,t_i$. The assumption $g_m\sigma\cdots\sigma g_1=\mathrm{Id}$ implies that $s_{m-1}=s_0$ and $t_{m-1}=t_0$, so we may choose $\alpha_{m-1}=\alpha_0$ and obtain $\tau_{m-1}=\tau_0$. We define
\[\theta_i:=\tau_{i}\sigma g_i\tau_{i-1}^{-1}\in\JJ,\quad i=1,\dots,m-2,\qquad \theta_{m-1}:=\tau_{m-1}(g_m\sigma g_{m-1})\tau_{m-2}^{-1}\in\JJ\]
and consider the sequence
\[\Id=\tau_0g_m\sigma g_{m-1}\sigma\cdots\sigma g_1\tau_0^{-1}=\tau_0g_m\tau_{m-1}^{-1}\tau_{m-1}\sigma g_{m-1}\tau_{m-2}^{-1}\tau_{m-2}\sigma g_{m-2}\dots \sigma g_1\tau_0^{-1}=\theta_{m-1}\theta_{m-2}\cdots\theta_1\]
The situation is visualised in the following commutative diagram:

\[\xymatrix{\ar@{-->}[r]^{\sigma g_1}\ar@{-->}[d]_{\tau_0}&\ar@{-->}[r]^{\sigma g_2}\ar@{-->}[d]_{\tau_1}&\ar@{-->}[d]^{\tau_2}\ar@{..>}[rr]&&\ar@{-->}[r]^{\sigma g_{n-1}}\ar@{-->}[d]_{\tau_{n-2}}&\ar@{-->}[r]^{\sigma g_n}\ar@{-->}[d]_{\tau_{n-1}}&\ar@{-->}[d]_{\tau_{n}}\ar@{..>}[rr]&&\ar@{-->}[r]^{g_m\sigma g_{m-1}}\ar@{-->}[d]^{\tau_{m-2}}&\ar@{-->}[d]^{\tau_{m-1}=\tau_0}\\
\ar@{-->}[r]_{\theta_1}&\ar@{-->}[r]_{\theta_2}&\ar@{..>}[rr]&&\ar@{-->}[r]_{\theta_{n-1}}&\ar@{-->}[r]_{\theta_n}&\ar@{..>}[rr]&&\ar@{-->}[r]_{\theta_{m-1}}&}
\]
We claim that for each $i=1,\dots,m-1$, the map $\theta_i\in\JJ$ is a quadratic map with only proper base-points in $\PP^2$. The only common base-point of the transformations $\sigma g_i$ and $\tau_{i}$ is $p_0$, hence by Lemma~\ref{lem:composition} the map $\tau_{i-1}g_i^{-1}\sigma$ is of degree $3$ and has $p_0,p_1,p_2,s_{i+1},t_{i+1}$ as base-points. The common base-points of the two transformations $\tau_{i-1}g_i^{-1}\sigma$ and $\tau_i$ are the points $p_0,s_{i+1},t_{i+1}$. Therefore, the composition $\theta_i=\tau_i\sigma g_i\tau_{i-1}^{-1}$ is quadratic and its base-points are $p_0$ and the image under $\tau_{i-1}$ of the two base-points of $\sigma g_{i}$ different from $p_0$. 
\[\xymatrix{\ar@{-->}[rr]^{\sigma g_i}^(.2){[p_0,p_1,p_2]}\ar@{-->}[d]_{\tau_{i-1}}_(.25){[p_0,s_{i-1},t_{i-1}]}&\qquad\qquad&\ar@{-->}[d]^{\tau_i}^(.25){[p_0,s_i,t_i]}\ar@{-->}[lld]|(.35){[p_0,p_1,p_2,s_i,t_i]}\\
\ar@{-->}[rr]^{\theta_i}&&}\]
Since $s_{i-1}$ and $t_{i-1}$ are general points, these images are proper points of $\PP^2$. So for each $i=1,\dots,m-1$, we find $\beta_i,\gamma_i\in\PGL_3(\kk)\cap\JJ$ such that $\theta_i=\beta_i\sigma\gamma_i$. 
\[\xymatrix{\ar@{-->}[r]^{\sigma g_1}\ar@{-->}[d]_{\tau_0}&\ar@{-->}[r]^{\sigma g_2}\ar@{-->}[d]_{\tau_1}&\ar@{-->}[d]^{\tau_2}\ar@{..>}[rr]&&\ar@{-->}[r]^{\sigma g_{n-1}}\ar@{-->}[d]_{\tau_{n-2}}&\ar@{-->}[r]^{\sigma g_n}\ar@{-->}[d]_{\tau_{n-1}}&\ar@{-->}[d]_{\tau_n}\ar@{..>}[rr]&&\ar@{-->}[r]^{g_m\sigma g_{m-1}}\ar[d]^{\tau_{m-1}}&\ar@{-->}[d]^{\tau_m=\tau_0}\\
\ar@{-->}[r]_{\beta_1\sigma\gamma_1}&\ar@{-->}[r]_{\beta_2\sigma\gamma_2}&\ar@{..>}[rr]&&\ar@{-->}[r]_(.2){\beta_{n-1}\sigma\gamma_{n-1}}&\ar@{-->}[r]_(.6){\beta_n\sigma\gamma_n}&\ar@{..>}[rr]&&\ar@{-->}[r]_{\beta_{m-1}\sigma\gamma_{m-1}}&}
\]
We write $\tilde{h}_1:=\gamma_1$, $\tilde{h}_{m}:=\beta_{m-1}$, $\tilde{h}_i:=\gamma_{i}\beta_{i-1}$, $i=2,\dots,m-1$.  
By Remark~\ref{rmk:composition}, we have
\[\deg(\tau_i\sigma g_i\cdots\sigma g_1)=\deg(\sigma g_i\cdots\sigma g_1)+1\]
and therefore
\[\deg(\theta_i\cdots\theta_1)=\deg(\tau_i\sigma g_i\cdots\sigma g_1\tau_0^{-1})=\deg(\sigma g_i\cdots \sigma g_1).\]
Hence we finally obtain
\[\deg(\sigma \tilde{h}_i\sigma\cdots\sigma\tilde{h}_1)=\deg(\beta_{i+1}^{-1}\theta_i\cdots\theta_1)=\deg(\sigma g_i\cdots \sigma g_1).\] 
This is part (2) of the lemma. By Lemma~\ref{lem:square}, all squares in the above diagram are generated by relations (\ref{rel1})--(\ref{rel5}), hence the whole diagram is generated by these relations. This is part (1). 

Let us look at part (3). 
The base-points of $(\theta_i\cdots \theta_1)^{-1}$ are $[1:0:0]$ and the proper images of the base-points of $(\sigma g_i\cdots\sigma g_1)^{-1}$ different from $[1:0:0]$. 
By construction of the $\tau_i$, none of the base-points of $(\sigma g_i\cdots\sigma g_1)^{-1}$ different from $[1:0:0]$ are on a line contracted by $\tau_i$. Moreover, the base-points of $(\sigma g_i\cdots\sigma g_1)^{-1}$ in the first neighbourhood of $[1:0:0]$ are sent by $\tau_i$ onto proper points of $\PP^2$. In conclusion, $(\theta_i\cdots\theta_1)^{-1}$ has less base-points infinitely near $[1:0:0]$. Induction step done.
The same argument yields part (4). 
\end{proof}

\begin{lem}\label{lem:no bp 2} 
Let $g_1,\dots,g_m\in\PGL_3(\kk)\cap\JJ$ be linear de Jonqui\`eres transformations. 
 Then there exist linear de Jonqui\`eres transformations $h_1,\dots,h_m\in\PGL_3(\kk)\cap\JJ$, and a de Jonqui\`eres transformation $\varphi\in\JJ$ such that 
\[\varphi g_m\sigma g_{m-1}\sigma\cdots \sigma g_1\varphi^{-1}=h_m\sigma h_{m-1}\sigma\cdots\sigma h_1\]
and such that the following properties hold:
\begin{enumerate}
\item the above relation is generated by relations $(\ref{rel1})$--$(\ref{rel5})$,
\item $\deg(\sigma h_i\cdots\sigma h_1)=\deg(\sigma g_i\cdots\sigma g_1)$ for all $i=1,\dots,m$,
\item $(\sigma h_i\sigma h_{i-1}\cdots \sigma h_1)^{-1}$ does not have any infinitely near base-points for all $i=1,\dots,m$.
\end{enumerate}
\end{lem}
\begin{proof}
Define $f_i:=\sigma g_i\sigma g_{i-1}\cdots \sigma g_1$ for $i=1,\dots,m$. After applying Lemma~\ref{lem:no bp 1} we may assume that none of the transformations $f_i^{-1}$ has base-points infinitely near to $[1:0:0]$.
We define 
\begin{align*}&N:=\max\{N(f_i,l)\mid i=1,\dots,m,\quad l\subset\PP^2\ \text{line}\}\\
& k:=\max\{i\mid N(f_i,l)=N\},\\
&\# :=\text{number of lines $l\subset\PP^2$ with $N(f_k,l)=N$}
\end{align*}
In other words, $N$ is the maximal discrepancy, $k$ is the maximal number of all $i$ such that $f_i$ has $N$ as a discrepancy and $\#\geq1$ is the number of lines contracted by $f_k$ with discrepancy $N$.
We do induction over the lexicographically ordered triple $(N,k,\#)$.

If $N=0$ there are no infinitely near base-points and we are done. Suppose that $N\geq1$.
We denote by $l_k\subset\PP^2$ a line such that $N=N(f_k,l_k)$ and define $p_k:=f_k(l_k)\in\PP^2$. 
Let $s\in\PP^2$ be a general point. 
Then there exists a linear de Jonqui\`eres map $h\in\PGL_3(\kk)\cap \JJ$ such that $\sigma h$ has $[1:0:0], p_k$ and  $s$ as base-points. 
By definition of $k$ and $N$ we have $N>N(f_{k+1},l_k)$. Since $f_{k+1}=\sigma g_{k+1}f_k$, Lemma~\ref{lem:discrepancy} implies that $p_k$ is also a base-point of $\sigma g_{k+1}$. It follows from Lemma~\ref{lem:composition} that $\deg(\sigma g_{k+1}(\sigma h)^{-1})=2$. Since $s$ is a general point, we find $a_1,a_2\in\PGL_3(\kk)\cap\JJ$ such that $\sigma g_{k+1}(\sigma h)^{-1}=a_2\sigma a_1$. 
Note that 
$N(\sigma h\sigma g_k\cdots\sigma g_1,l_k)=N-1$ by Lemma~\ref{lem:discrepancy}. 

If $N(\sigma g_{k-1}\cdots\sigma g_1,l_k)<N$, we  find, by similar arguments as above, $b_1,b_2\in\PGL_3(\kk)\cap\JJ$ such that $\sigma h\sigma g_k=b_2\sigma b_1$ 
\[\xymatrix{\ar@{-->}[r]^{\sigma g_1}&\ar@{-->}[r]^{\sigma g_2}&\cdots&\cdots 
\ar@{-->}[r]^{\sigma g_{k-1}}&\ar@{-->}[dr]_{b_2\sigma b_1} \ar@{-->}[r]^{\sigma g_k}&\ar@{-->}[d]^{\sigma h}\ar@{-->}[r]^{\sigma g_{k+1}}  & \ar@{-->}[r]^{\sigma g_{k+2}}&\cdots \\
&&&&&\ar@{-->}[ur]_{a_2\sigma a_1}&}\]
The new sequence $g_m\sigma g_{m-1}\cdots \sigma g_{k+2}a_2\sigma (a_1b_2)\sigma (b_1g_{k-1})\sigma g_{k-2}\cdots\sigma g_1$ has maximal discrepancy at most $N$. If it is $N$, that means that there is either another line $l\subset\PP^2$ such that $N(\sigma g_k\cdots\sigma g_1,l)=N$, in which case $\#$ has decreased. Or that $N(\sigma g_k\cdots\sigma g_1,l)<N$ for all lines $l$ and $N(\sigma g_i\cdots\sigma g_1,l_i)=N$ for some $i<k$ and some line $l_i$. In any case, the triple $(N,k,\#)$ decreases. Note as well that 
\[\deg(\sigma h\sigma g_k\cdots\sigma g_1)=\deg(\sigma g_k\cdots \sigma g_1).\]

If $N(\sigma g_{k-1}\cdots\sigma g_1,l_k)=N$, 

 there are two options by Lemma~\ref{lem:discrepancy}:
\begin{itemize}
\item[(a)] The point $p_k$ is a base-point of $(\sigma g_k)^{-1}$ and, if $\eta\colon S\rightarrow \PP^2$ is the blow-up of the base-points of $(\sigma g_k)^{-1}$, then $(\eta^{-1} f_k)^{-1}(l_k)$ is the intersection point of an exceptional divisor and the strict transform of a line contracted by $(\sigma g_k)^{-1}$. 
\item[(b)] The point $p_k$ is not a base-point of $(\sigma g_k)^{-1}$ and is not on its contracted lines.
\end{itemize}

In case (a) we can proceed analogously as in the case above, where $N(\sigma g_{k-1}\cdots\sigma g_1,l_k)<N$.
So assume now that we are in case (b). The image $p_k'$ of $p_k$ by $(\sigma g_k)^{-1}$ is a proper point of $\PP^2$. 
The point $s\in\PP^2$ is general, hence also its image $s'$ by $(\sigma g_k)^{-1}$ is a general point of $\PP^2$. In particular, there exists $h'\in\PGL_3(\kk)\cap \JJ$ such that $\sigma h'$ has base-points $[1:0:0],p_k'$ and $s'$. 
The map $\sigma h\sigma g_k(\sigma h')^{-1}$ is of degree $2$, and since the points $s$ and $s'$ are in general position, we find $c_1,c_2\in\PGL_3(\kk)\cap\JJ$ such that $\sigma h\sigma g_k(\sigma h')^{-1}=c_2\sigma c_1$. 
Moreover, we have $N(\sigma h'\sigma g_{k-1}\cdots \sigma g_1,l_k)=N-1$. Also note that 
\[\deg(\sigma h'\sigma g_{k-1}\cdots\sigma g_1)=\deg(\sigma g_{k-1}\cdots\sigma g_1).\]
We proceed like this until we find an index $i<k$ such that $N(\sigma g_i\cdots\sigma g_1,l_k)<N$. Such an index has to exist because $\sigma g_m\cdots\sigma g_1=\id$.
There, the preimage of $p_k$ is a base-point of $(\sigma g_i\cdots\sigma g_1)^{-1}$, and we repeat the construction of the former case. It yields a new sequence with maximal discrepancy at most $N$. 
If it is equal to $N$ then either at the index $k$ and another line or at an index smaller than $k$. In other words, the triple $(N,k,\#)$ has decreased. 

The commutative diagram below illustrates the procedure if $N(\sigma g_{k-2}\cdots\sigma g_1,l_k)<N$.
\[\xymatrix{\ar@{-->}[r]^{\sigma g_1}&\ar@{-->}[r]^{\sigma g_2}&\cdots&\cdots 
\ar@{-->}[r]^{\sigma g_{k-1}} \ar@{-->}[dr]_{b_2\sigma b_1}&\ar@{-->}[d]^{\sigma h'}\ar@{-->}[r]^{\sigma g_k}&\ar@{-->}[d]^{\sigma h}\ar@{-->}[r]^{\sigma g_{k+1}}  & \ar@{-->}[r]^{\sigma g_{k+2}}&\cdots \\
&&&&\ar@{-->}[r]_{c_2\sigma c_1}&\ar@{-->}[ur]_{a_2\sigma a_1}&}\]

It remains to remark that all the squares and triangles are generated by relations (\ref{rel1})--(\ref{rel5}) by Lemma~\ref{lem:deg2} and Lemma~\ref{lem:square}. 
\end{proof}

\section{Proof of Theorem~\ref{thm:relations}}

\begin{prop}\label{prop:relations J}
Let $g_1,\dots,g_m\in\mathrm{PGL}_3(\kk)\cap\JJ$ and suppose that $\Id=g_m\sigma g_{m-1}\sigma\cdots\sigma g_1$ as maps. Then this expression is generated by relations $(\ref{rel1})$--$(\ref{rel5})$.\par
\end{prop}
\begin{proof}
Let $\Lambda_0$ be the complete linear system of lines in $\PP^2$ and define 
\[\Lambda_i:=\sigma g_{i-1}\sigma\cdots\sigma g_1(\Lambda_0),\quad i=1,\dots,m.\]
Let 
\[\delta_i:=\deg(\Lambda_i),\quad D:=\max\{\delta_i\mid i=1,\dots, m\}\ \text{and}\ n:=\max\{i\mid \delta_i=D\}.\]
We do induction on the lexicographically ordered set of pairs of positive integers $(D,n)$.\par
If $D=1$, then $m=1$ and there is nothing to prove.\par
Let $D>1$. By Lemma~\ref{lem:no bp 2}, we can suppose that for each $i=1,\dots,m$ the transformation $(g_i\sigma g_{i-1}\sigma\cdots \sigma g_1)^{-1}$ does not have any infinitely near base-points, and we can do this without increasing the pair $(D,n)$. Equivalently, each $\Lambda_i$ does not have any infinitely near base-points.

All the $g_i$ are de Jonqui\`eres maps and therefore fix $[1:0:0]$, the maps $\sigma g_i$ and $\sigma $ always have $[1:0:0]$ as common base-point. 
In particular, Lemma~\ref{lem:composition} yields $\deg(\sigma g_i\sigma)\leq3$ for all $i=1,\dots,m$. 
We will look at the three distinct cases $\deg(\sigma g_n\sigma)=1,2$
 and $3$.\par 
{\bf(a)} If $\deg(\sigma g_n\sigma)=1$, Lemma~\ref{lem:deg1} implies that we can replace the word $\sigma g_n\sigma$ by the linear map $g_n'=\sigma g_n\sigma$ using only relations  (\ref{rel1})--(\ref{rel5}). We obtain a new pair $(D',n')$ where $D'\leq D$ and if $D=D'$ then $n<n'$. \par
{\bf(b)} If $\deg(\sigma g_n\sigma)=2$, then $\sigma$ and $\sigma g_n$ have exactly two common base-points, among them $[1:0:0]$. Up to coordinate permutations, these two points are $[1:0:0]$ and $[0:1:0]$. More precisely, there exist two permutations of coordinates $\tau_1, \tau_2\in S_3\cap\JJ$ such that $\tau_1g_n\tau_2$ fixes the points $[1:0:0]$ and $[0:1:0]$. Hence there are $a_1, a_2, b_1, b_2, c\in\kk$ such that
\[\tau_1g_n\tau_2=[a_1x+a_2z:b_1y+b_2z:cz].\]
Using relations (\ref{rel1}) and (\ref{rel3}) we get
\[g_m\sigma\cdots \sigma g_{n+1}\sigma \tau_1^{-1}\tau_1g_n\tau_2\tau_2^{-1}\sigma\cdots g_1=g_m\sigma\cdots\sigma (g_{n+1}\tau_1^{-1})\sigma(\tau_1g_n\tau_2)\sigma(\tau_2^{-1}g_{n-1})\sigma\cdots g_1.\]
This replacement does not change the pair $(D,n)$. So, we may assume that 
\[\tau_1=\tau_2=\mathrm{id}\quad\text{and}\quad g_n=[a_1x+a_2z:b_1y+b_2z:cz].\]
By assumption, for $i=1,\dots,n$, the maps $g_i\sigma g_{i-1}\sigma\cdots \sigma g_1$ have no infinitely near base-points. It follows that $\Lambda_n$ has no infinitely near base-points. 

\indent {\bf(b1)} If $a_2b_2\neq0$, then no three of the base-points of $\sigma$ and $\sigma g_n$ are collinear. By Lemma~\ref{lem:deg2} there exist $g',g''\in\mathrm{PGL}_3(\kk)$ such that we can replace the word $\sigma g_n\sigma$ with the word $g'\sigma g''$ using relations (\ref{rel1})--(\ref{rel5}). This yields a new pair $(D',n')$ where $D'\leq D$ and if $D=D'$ then $n'<n$. 

{\bf(b2)} We want to see that $a_2b_2=0$ is impossible. Suppose that $a_2b_2=0$. Then $q:=g_{n}^{-1}([0:0:1])$ is a base-point of $\sigma g_{n}$ on a line that is contracted by $(\sigma g_{n-1})^{-1}$ (see Remark~\ref{rmk:the linear map}).
Remark~\ref{rmk:composition} implies that 
\[D-1=\delta_{n+1}=D+1-m_{[0:1:0]}(\Lambda_n)-m_q(\Lambda_n)\]
In particular, we have 
\[m_q(\Lambda_n)=1\]
Since it is not a base-point of $\sigma$, it is not a base-point of $(\sigma g_{n-1})^{-1}$. Hence its proper image by $(\sigma g_{n-1})^{-1}$ is a base-point of $\Lambda_{n-1}$. Because of $a_2b_2=0$, this point is an infinitely near point, a contradiction to our assumption that the $\Lambda_i$ do not have any infinitely near base-points. 

\indent{\bf(c)} Assume that $\deg(\sigma g_n\sigma)=3$, so $\sigma$ and $\sigma g_n$ have one common base point, which is $[1:0:0]$. Denote by $p_0=[1:0:0]$, $p_1=[0:1:0]$, $p_2=[0:0:1]$ the base-points of $\sigma$ and by $p_0,q_1,q_2$ the base-points of $\sigma g_n$. Remark~\ref{rmk:composition} implies
\begin{align*}
&\delta_n\geq \deg(\Lambda_{n-1})=\deg((\sigma g_{n-1})^{-1}(\Lambda_n))=\delta_n+1-m_{p_1}(\Lambda_n)-m_{p_2}(\Lambda_n)\\
&\delta_n>\deg(\Lambda_{n+1})=\deg(\sigma g_n(\Lambda_n))=\delta_n+1-m_{q_1}(\Lambda_n)-m_{q_2}(\Lambda_n).
\end{align*}
Therefore, we obtain
\begin{equation*}1=m_{q_1}(\Lambda_n)=m_{q_2}(\Lambda_n),\quad 1\leq m_{p_1}(\Lambda_n)+m_{p_2}(\Lambda_n)\end{equation*}
Choose $i\in\{1,2\}$ such that $m_{p_i}(\Lambda_n)=1$. Then the points $[1:0:0],p_i$ and $q_1$ are not collinear (because $m_{p_0}(\Lambda_n)+m_{p_1}(\Lambda_n)+m_{q_1}(\Lambda_n)>\delta_n$) and there exists a $g'\in\PGL_3(\kk)\cap\mathcal{J}_*$ such that $\sigma g'$ has base-points $[1:0:0],p_i$ and $q_1$. Consider Figure~\ref{fig:d6}; we claim that the left pair and the right pair of maps satisfy the assumptions of case (b1), which implies that $\sigma g(\sigma g_{n-1}')^{-1}$ and $\sigma g_n\sigma g'$ can be replaced by maps of the form $h\sigma h'$ and $h''\sigma h'''$ respectively for some linear maps $h,h',h'', h'''$ and the pair $(D,n)$ decreases. \par
\begin{minipage}[h]{\textwidth}
\begin{center}
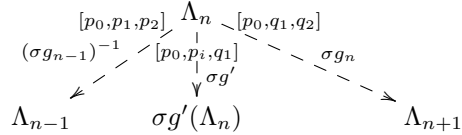

$\xymatrix{&\Lambda_n\ar@{-->}[ld]_(.25){[p_0,p_1,p_2]}_(.55){(\sigma g_{n-1})^{-1}}\ar@{-->}[d]|(.37){[p_0,p_i,q_1]}^(.6){\sigma g'}\ar@{-->}[rrd]^(.25){[p_0,q_1,q_2]}^(.55){\sigma g_n}&&\\ \Lambda_{n-1}&\sigma g'(\Lambda_n)&&\Lambda_{n+1}& }$
\captionof{figure}{The two maps on the left and the two maps on the right satisfy the assumptions of case (b1).}\label{fig:d6}
\end{center}
\vspace{.2em}
\end{minipage}
The maps $\sigma$ and $\sigma g'$ have two common base-points, namely $p_0$ and $p_i$. The maps $\sigma g'$ and $\sigma g_n$ have the base-points $p_0$ and $q_1$ in common. Remark~\ref{rmk:composition} tells us that
\[\deg(\sigma g'(\Lambda_n))=\delta_n+1-m_{p_i}(\Lambda_n)-m_{q_1}(\Lambda_n)<\delta_n=D.\]
It remains to check that no three of the four points $p_0,p_1,p_2,q_1$ are collinear and that no three of the four points $p_0,p_i,q_1,q_2$ are collinear. Indeed, in the latter case all four points are base-points of $\Lambda_n$ and the image of $p_i$ by $\sigma g_n$ is a base-point of $\Lambda_{n+1}$, which does not have any infinitely near base-points. Therefore, no three of $p_0,p_i,q_1,q_2$ are collinear. In the first case, at least $p_0,p_i,q_1$ are base-points of $\Lambda_n$ and therefore not collinear by Remark~\ref{rmk:composition}. The points $p_1,p_2,q_1$ and the points $p_0,p_2,q_1$ are not collinear because $\Lambda_{n-1}$ has no infinitely near base-points. 
Thus, we can apply case (b1) to the maps $\sigma$ and $\sigma g'$ and to the maps $\sigma$ and $\sigma g_ng'^{-1}$. The pair $(D,n)$ decreases. 
\end{proof}

So far we have been working with de Jonqui\`eres maps only. Now we will use the structure of the plane Cremona group as an almost amalgamated product from Theorem \ref{blancisk} to prove the Main Theorem:

\begin{proof}[Proof of Main Theorem]
Let $G=\langle \sigma,\PGL_3(\kk)\mid (\ref{rel1})-(\ref{rel5})\rangle$ be the group generated by $\sigma$ and $\PGL_3(\kk)$ divided by the relations (\ref{rel1})--(\ref{rel5}), and let $\pi\colon G\rightarrow\Bir(\PP^2)$ be the canonical homomorphism that sends generators onto generators. It follows from Proposition~\ref{prop:relations J} that sending an element of $\JJ$ onto its corresponding word in $G$ is well defined. This yields a homomorphism $w\colon\JJ\longrightarrow G$ that satisfies $\pi\circ w=\mathrm{Id}_{\JJ}$ and is therefore injective. Consider the commutative diagram
\[\xymatrix{G&\JJ\ar@{_{(}->}[l]_{w}\\ \PGL_3(\kk)\ar@{^{(}->}[u]&\PGL_3(\kk)\cap\JJ\ar@{_{(}->}[l]\ar@{_{(}->}[u]}\]
where all unnamed homomorphisms are the canonical inclusions. The universal property of the amalgamated product implies that there exists a unique homomorphism 
\[\varphi\colon\PGL_3(\kk)\ast_{\PGL_3(\kk)\cap\JJ}\JJ\rightarrow G\] 
such that the following diagram commutes:
\[\xymatrix{G&&\\
&\PGL_3(\kk)\ast_{\PGL_3(\kk)\cap\JJ}\JJ \ar[lu]_{\exists\,!\ \varphi}&\JJ\ar[l]_(.3){w}\ar@/_1pc/[ull]\\ 
&\PGL_3(\kk)\ar@{^{(}->}[u] \ar@/^1.5pc/[ulu]&\PGL_3(\kk)\cap\JJ\ar@{_{(}->}[l]\ar@{_{(}->}[u]}\]
By Theorem~\ref{blancisk}, the group $\Bir(\PP^2)$ is isomorphic to $\PGL_3(\kk)\ast_{\PGL_3(\kk)\cap\JJ}\JJ$ divided by the relation $\tau\sigma\tau\sigma$, where $\tau=[y:x:z]$, which is a relation that holds as well in $G$. So $\varphi$ factors through the quotient $\PGL_3(\kk)\ast_{\PGL_3(\kk)\cap\JJ}\JJ/\langle\tau\sigma\tau\sigma\rangle$ and we thus obtain a homomorphism $\bar{\varphi}\colon \Bir(\PP^2)\to G$. In fact, the homomorphisms $\pi$ and $\bar{\varphi}$ both send generators to generators:
\begin{align*}G\stackrel{\pi}\longrightarrow\Bir(\PP^2),\quad &\sigma\longmapsto\sigma,\quad\alpha\longmapsto\alpha\quad\forall\ \alpha\in\mathrm{PGL}_3\\\
\Bir(\PP^2)\stackrel{\bar{\varphi}}\longrightarrow G,\quad&\sigma\mapsfrom\sigma,\quad\alpha\longmapsto\alpha\quad\forall\ \alpha\in\mathrm{PGL}_3.
\end{align*}
Hence $\bar{\varphi}$ and $\pi$ are isomorphisms that are inverse to each other.
\end{proof}

\vfill
\end{document}